\documentclass{amsart}

\usepackage{amsmath}
\usepackage{graphicx}
\usepackage{amssymb}
\usepackage{epstopdf}
\usepackage{tikz-cd}
\usepackage{mathrsfs}
\usepackage{xcolor}
\usepackage{hyperref}
\usepackage{tikz}
\usepackage{tikz-cd} 
\usetikzlibrary{cd}
\usepackage{framed}
\usepackage[T1]{fontenc}
\usepackage{bm}
\usepackage{comment}
\usepackage[normalem]{ulem}
\usepackage[left]{lineno}

\numberwithin{equation}{section}

\newtheorem{thm}{Theorem}[section]
\newtheorem{letterthm}{Theorem}

\newtheorem{letterconjecture}{Question}

\newtheorem{lem}[thm]{Lemma}
\newtheorem{prop}[thm]{Proposition}
\newtheorem{defn}[thm]{Definition}
\newtheorem{cor}[thm]{Corollary}

\newtheorem{remark}[thm]{Remark}

\begin{document}

\title[Derived Hecke algebras]
{Non-vanishing mod $p$ of derived Hecke algebras of multiplicative groups over number fields, II}
\author[Jaesung Kwon]{JAESUNG KWON${}^1$}
\thanks{${}^1$This work was supported by Basic Science Research Program 
through the National Research Foundation of Korea (NRF) 
funded by the Ministry of Education (RS-2025-00563144).
The author thank Dohyeong Kim for helpful discussions and comments.}
\email{js.kwon@unist.ac.kr}
\address{Ulsan National Institute of Science and Technology, 
50, UNIST-gil, Eonyang-eup, Ulju-gun, Ulsan, Republic of Korea, 44919}
\date{\today}

\subjclass[2020]{11F75, 20C08}
\keywords{}

\begin{abstract} 
We investigate the derived Hecke action on the cohomology of an arithmetic manifold 
associated to the multiplicative group over a number field. 
Under mild assumptions, we prove that the derived Hecke action is non-vanishing modulo $p$ in every degree.
Moreover, the total cohomology ring with $\mathbb F_p$-coefficients
is a free graded module of rank one over the derived Hecke algebra, generated in degree zero.
\end{abstract}
\maketitle
\numberwithin{equation}{section}
\setcounter{tocdepth}{1}
\tableofcontents

\section{Introduction}
Let $F$ be a number field and $\mathbf{G}$ a reductive algebraic group over $F$. 
Let $\mathbb{A}_F^{(\infty)}$ be the ring of finite adeles of $F$, 
$S_\infty$ the quotient of $\mathbf{G}(F\otimes_\mathbb{Q}\mathbb{R})$ 
by a maximal compact connected subgroup of itself, 
and $K$ an open compact subgroup of $\mathbf{G}(\mathbb{A}_F^{(\infty)})$. 
Then, the associated quotient
\begin{equation}\label{eq:am}
\mathbf{G}(F)\backslash \left(S_\infty\times\mathbf{G}(\mathbb{A}_F^{(\infty)})\right)/K
\end{equation}
is known to be homeomorphic to a finite union of locally symmetric spaces.

Spaces like \eqref{eq:am} are special in that their singular cohomology groups are modules over Hecke algebras.
We briefly recall their definition.
Let $R$ be a commutative ring with unity and $v$ a prime of $F$. Then, the {\it abstract Hecke algebra at $v$} is defined as
\begin{equation}\label{abstract:classcial:Hecke:algebra:def}
\mathcal{H}_v^0:=\mathrm{Hom}(R[\mathbf{G}(F_v)/K_v],R[\mathbf{G}(F_v)/K_v])
\end{equation}
where $F_v$ is the completion of $F$ at $v$ and $K_v$ is the $v$-component of $K$.
The multiplication on $\mathcal{H}_v^0$ is given by convolution and the $\mathrm{Hom}$-set is taken in a suitable category of $\mathbf{G}(F_v)$-modules.

Venkatesh \cite{venkatesh2016derived} introduced the derived version of \eqref{abstract:classcial:Hecke:algebra:def}:
\begin{defn}[Abstract derived Hecke algebra, Venkatesh {\cite[Definition 2.2]{venkatesh2016derived}}] 
For a prime $v$ of $F$, define
\begin{equation*}
\mathcal{H}_v^\bullet:=\mathrm{Ext}^\bullet_{C_{\rm c}(\mathbf{G}(F_v),R)}(R[\mathbf{G}(F_v)/K_v],R[\mathbf{G}(F_v)/K_v]),
\end{equation*}
where $C_{\rm c}(\mathbf{G}(F_v),R)$ is the ring of compactly supported locally constant $R$-valued functions on $\mathbf{G}(F_v)$.
\end{defn}

Furthermore, the graded algebra $\mathcal{H}_v^\bullet$ acts on the total cohomology space of \eqref{eq:am} 
graded by the cohomological degrees.
The {\it derived Hecke algebra} $\mathbb{T}_R^\bullet$ is defined to be the algebra 
generated by the graded endomorphisms induced by $\mathcal{H}_v^\bullet$ 
as $v$ ranges over all unramified primes) and $\mathcal{H}_v^0$
on the total cohomology $H^\bullet(Y_K,R)$, 
where $Y_K$ is the arithmetic manifold defined in \eqref{eq:am}. 

With the choice $\mathbf{G}:=\mathbb{G}_{\mathrm{m}/F}$, we describe the counterpart of \eqref{eq:am}.
Let $F_\mathbb{R}:=F\otimes_\mathbb{Q}\mathbb{R}$ be the real Minkowski space over $F$. Let $F_{\mathbb{R},+}^\times$ be the connected component of $1\in F_\mathbb{R}^\times$. 
Put
\begin{equation*}
U:=\{x=(x_\tau)_\tau\in F^\times_{\mathbb{R},+}: |x_\tau|=1\text{ for any infinite places }\tau\text{ of }F\},
\end{equation*}
which is the maximal compact connected subgroup of $F_{\mathbb{R},+}^\times$.
For an integral ideal $\mathfrak{N}$ of $F$, put
\begin{equation*}
K(\mathfrak{N}):=\left(\prod_{v\mid\mathfrak{N}}1+\mathfrak{N}O_v\right)\times\prod_{v\nmid\mathfrak{N}} O_v^\times
\end{equation*} 
and define
\begin{align}\label{eq:torus}
Y(\mathfrak{N}):=F^\times\backslash\left((F_\mathbb{R}^\times/U)\times\mathbb{A}_F^{(\infty),\times}\right)/K(\mathfrak{N}).
\end{align}
We will call $Y(\mathfrak N)$ an \emph{arithmetic torus}.
By~\cite[Proposition 3.6]{KimKwon2024}, $\mathbb{T}^\bullet$ is a graded commutative $\mathbb T^0$-algebra.
So we may consider the following map of graded $\mathbb T^\bullet$-modules 
\begin{align}\label{eq:psi}
\Psi_R^\bullet:\mathbb{T}^\bullet\otimes_{\mathbb{T}^0} H^0(Y(\mathfrak{N}),R)
\longrightarrow H^\bullet(Y(\mathfrak{N}),R),\quad h\otimes c\mapsto hc,
\end{align}
which is induced by the action of the derived Hecke algebra.
The conjecture of Prasanna--Venkatesh 
\cite[\textsection\,5.4, Main Conjecture]{prasanna2021automorphic} motivates
studying when $\Psi_R^\bullet$ is an isomorphism.

We specialize to finite-field coefficients.
Let $p$ be a rational prime. 
Denote by $\mathbb{T}^\bullet=\mathbb{T}^\bullet_{\mathbb{F}_p}$ 
the derived Hecke algebra of $Y(\mathfrak{N})$ with coefficients in $\mathbb{F}_p$.
As mentioned earlier, $\mathbb{T}^\bullet$ acts on the cohomology ring 
$H^\bullet(Y(\mathfrak{N}),\mathbb{F}_p)$. 
Write $\Psi^\bullet=\Psi_{\mathbb F_p}^\bullet$ 
for the specialization of \eqref{eq:psi} to $R=\mathbb F_p$.

Let $\mathcal{O}_F$ be the ring of integers of $F$. 
Put 
\[
E(\mathfrak{N}):=F^\times\cap F_{\mathbb{R},+}^\times\cap K(\mathfrak{N}),
\]
which is the subgroup of totally positive elements of $\mathcal{O}_F^\times$ 
congruent to $1$ modulo $\mathfrak{N}$.
From now on, we assume that $E(\mathfrak{N})$ is torsion free, 
which holds if the norm of $\mathfrak{N}$ is sufficiently large.
Under this assumption, Kim and Kwon \cite[Theorem A]{KimKwon2024} 
proved that $\Psi^1$ is an isomorphism of $\mathbb T^0$-modules for all but finitely many $p$.
Motivated by this degree one result, we propose the following mod $p$ analogue of the Prasanna--Venkatesh conjecture:
\begin{letterconjecture}\label{conj:pv:finite}
For all but finitely many $p$,
the map $\Psi^\bullet$ in \eqref{eq:psi} is an isomorphism of graded $\mathbb{T}^\bullet$-modules.
\end{letterconjecture}
This holds automatically when the higher degree cohomology groups are vanishing.
Therefore, we will assume from now on that $F$ is neither $\mathbb Q$ nor an imaginary quadratic field.
In this paper, we answer Question~\ref{conj:pv:finite} by extending the result of \cite{KimKwon2024}:

\begin{letterthm}\label{thm:A}
If $p$ does not divide the order of $\mathcal{O}_F^\times/E(\mathfrak{N})$, 
then $\Psi^\bullet$ is an isomorphism of $\mathbb{T}^\bullet$-modules.
\end{letterthm}

Note that the assumption on $p$ in Theorem~\ref{thm:A} fails for only finitely many primes. 

Next, we prove the spectral degeneracy for general degree 
which extends the result of Kim--Kwon (\cite[Theorem 4.9]{KimKwon2024}):

\begin{letterthm}\label{thm:B}
If $p$ does not divide the order of $\mathcal{O}_F^\times/E(\mathfrak{N})$, 
then for an integer $j>0$, 
a Hecke eigensystem occurs in $H^j(Y(\mathfrak{N}),\overline{\mathbb{F}}_p)$ 
if and only if it occurs in $H^0(Y(\mathfrak{N}),\overline{\mathbb{F}}_p)$.
\end{letterthm}

Finally, we give the following description of the $\mathbb{T}^\bullet$-module structure of the cohomology ring
$H^\bullet(Y(\mathfrak{N}),\mathbb{F}_p)$:

\begin{letterthm}\label{thm:C}
If $p$ does not divide the order of $\mathcal{O}_F^\times/E(\mathfrak{N})$, then $H^\bullet(Y(\mathfrak{N}),\mathbb{F}_p)$ is a free graded $\mathbb{T}^\bullet$-module of rank one, generated in degree zero.
\end{letterthm}

The key ingredients are the formula of defived Hecke operators 
(Proposition~\ref{computation:hecke:operator:product}), 
the injectivity of $\Psi_R^\bullet$ (Proposition~\ref{injectivity:psi}),
dimension formulas (Proposition~\ref{dimension:prop} and Lemma~\ref{dimension:tp:dp}), 
and the result of the previous work~\cite{KimKwon2024} for degree one case.

To conclude the introduction, we mention a variant \cite{Harris2019} of the aforementioned conjecture \cite[\textsection\,5.4 Main Conj.]{prasanna2021automorphic}, also known as the weight one form case, where units in number fields play a central role. Its compatibility with the Stark conjecture is established in \cite{zhang2023}. See \cite{Darmon2022} and \cite{Zhang2024} for its resolution in the case of most dihedral forms.

\subsection{Comparison with previous work}

Note that in \cite[Proposition 4.2]{KimKwon2024}, injectivity was proved only for
the degree-one map $\Psi_{\mathbb F_p}^1$ with coefficients in $\mathbb F_p$.
In the present paper, we prove that $\Psi_R^\bullet$ is injective in all degrees 
for arbitrary coefficient rings $R$ (Proposition~\ref{injectivity:psi}).
The proof uses the injectivity of the restriction map to the cohomology group of degree zero 
(Proposition~\ref{hecke:action:faithful}), together with the module structures of $H^0(Y(\mathfrak N),R)$,
$\mathbb T_R^0$, and $\mathbb T_R^\bullet$ 
(Prosposition~\ref{Hecke:module:structure} and~\ref{hecke:zero:free}).

We also strengthen the formulation of spectral degeneracy in
\cite[Theorem 4.9]{KimKwon2024}.
Under the hypothesis of Theorem~\ref{thm:A}, for each Hecke eigensystem occurring in
$H^0(Y(\mathfrak N),\overline{\mathbb F}_p)$, the map induced by $\Psi^j$ on its eigenspace
is nonzero for some degree $j$ in which the cohomology is nonzero
(Theorem~\ref{spectral:degeneracy}).
This is stronger than merely asserting that the eigensystem also occurs in degree $j$:
the derived Hecke action produces a nonzero class in that eigenspace from degree zero.

We correct an error in the statement of \cite[Proposition 3.7 (1)]{KimKwon2024}, 
which asserts that $H^0(Y(\mathfrak N),R)=\mathbb T_R^01_1$. 
Here, $1_1$ is a cohomology class of degree zero, 
which is represented by the characteristic function of the identity component of $Y(\mathfrak N)$.
In \emph{loc. cit.}, $\mathbb T^0_R$ is generated by operators at unramified primes $v$ for which $N(v)$ is invertible in $R$.
For $R=\mathbb Z$, this algebra consists only of scalar-multiplication operators, 
so the assertion fails whenever $Y(\mathfrak N)$ has more than one connected component.
Therefore, in this paper, 
we define $\mathbb T^0_R$ to be the $R$-algebra generated by the degree zero Hecke operators at every unramified primes $v$.
With this definition, we can show that $H^0(Y(\mathfrak N),R)=\mathbb T_R^01_1$.
Furthermore, we show that $H^0(Y(\mathfrak N),R)$ is freely generated by $1_1$ over $\mathbb T_R^0$
(Proposition~\ref{hecke:zero:free}).

\subsection{Outline of the paper}

In \textsection \ref{section:torus}, we recall
some basic properties of $Y(\mathfrak{N})$. For example, we show that our arithmetic torus $Y(\mathfrak N)$
is homotopy equivalent to a disjoint union of $r$-tori, where $r$ is the rank of the unit group of $\mathcal{O}_F$.

In \textsection \ref{section:hecke}, we recall necessary definitions and results from \cite{venkatesh2016derived}. 
In particular, explicit descriptions of our derived Hecke operators are given, 
which will be used in later sections. Also we prove the injectivity of $\Psi_R^\bullet$.

In \textsection \ref{section:nonvanishing}, we specialize our interest to the case of $R=\mathbb{F}_p$ 
and investigate the map $\Psi^\bullet$. At the end of the section, 
Theorems~\ref{thm:A},~\ref{thm:B} and~\ref{thm:C} will be proved.

\section{Arithmetic torus}\label{section:torus}
Recall from \eqref{eq:torus} that $Y(\mathfrak N)$ denotes the arithmetic torus 
associated to an integral ideal $\mathfrak N$. Here we set up basic notation related to $Y(\mathfrak N)$.
Let $r_1$ and $r_2$ be the number of real and complex places of $F$, respectively. 
Let $r$ be the rank of $\mathcal{O}_F^\times$, which is equal to $r_1+r_2-1$ by the Dirichlet unit theorem.
Recall that
$E(\mathfrak{N})=F^\times\cap F_{\mathbb{R},+}^\times\cap K(\mathfrak{N})$.
From now on, we assume that $E(\mathfrak{N})$ is torsion free.

\begin{prop}\label{property:torus}
The space $E(\mathfrak{N})\backslash F_{\mathbb{R},+}^\times/U$ 
is a model for the classifying space of $E(\mathfrak{N})$.
\end{prop}

\begin{proof}
By the proof of \cite[Proposition 2.1 (3)]{KimKwon2024}, the regulator map identifies
the space $E(\mathfrak{N})\backslash F_{\mathbb{R},+}^\times/U$ 
with the product of $\mathbb R$ and an $r$-torus whose first homotopy group is isomorphic to $E(\mathfrak N)$. 
Therefore $E(\mathfrak{N})\backslash F_{\mathbb{R},+}^\times/U$ is a $K(E(\mathfrak{N}),1)$.
As $E(\mathfrak{N})$ is discrete, this proves the assertion.
\end{proof}

Let $\mathbb{A}_F$ be the adele ring of $F$. For $a\in \mathbb{A}_F^\times$, let us denote by $a_\infty$ and $a^{(\infty)}$ the Archimedean and the non-Archimedean part of $a$, respectively.
For $a\in\mathbb{A}_F^{\times}$, let us define a map $\iota_a:E(\mathfrak{N})\backslash F_{\mathbb{R},+}^\times/U\to Y(\mathfrak{N})$ by
\begin{equation*}
\iota_a \colon x\longmapsto (a_\infty x,a^{(\infty)}).
\end{equation*}
For each $a\in\mathbb{A}_F^{\times}$, $\iota_a$ is a homeomorphism onto its image 
({\cite[Proposition 2.2]{KimKwon2024}}).

Let us denote by $\mathrm{Cl}_F^+(\mathfrak{N})$ the idelic narrow ray class group of $F$ of modulus 
$\mathfrak{N}$, which is given by
\begin{equation*}
\mathrm{Cl}_F^+(\mathfrak{N}):=F^\times\backslash\mathbb{A}_F^\times/\left(F_{\mathbb{R},+}^\times\times K(\mathfrak{N})\right).
\end{equation*}
Let us denote by $Y(\mathfrak{N})^\circ$ the image of $\iota_1$, 
which is equipped with the group structure and is homotopy equivalent to an $r$-torus. 
It is a standard fact that the map
\begin{equation*}
Y(\mathfrak{N})\longrightarrow\mathrm{Cl}_F^+(\mathfrak{N}),\quad (x,a)\longmapsto ax
\end{equation*}
is a principal $Y(\mathfrak{N})^\circ$-bundle, 
where the space $\mathrm{Cl}_F^+(\mathfrak{N})$ is equipped with the discrete topology.
Furthermore, the fibre of the bundle over an ideal class $[a]\in\mathrm{Cl}_F^+(\mathfrak{N})$ 
is given by the image of the map $\iota_a$ 
(\cite[Proposition 2.3]{KimKwon2024}). From this, we obtain the following:

\begin{prop}\label{torus:copy}
The arithmetic torus $Y(\mathfrak{N})$ is homeomorphic to the $\mathrm{Cl}_F^+(\mathfrak{N})$-disjoint union of 
$E(\mathfrak{N})\backslash F^\times_{\mathbb{R},+}/U$ 
under the map $\bigsqcup_{[a]\in\mathrm{Cl}_F^+(\mathfrak{N})}\iota_a$.
\end{prop}

\begin{lem}\label{representative:choice} 
The image of $\iota_a$ is independent of the choice of a representative of a ray class 
$[a]\in\mathrm{Cl}_F^+(\mathfrak{N})$.
\end{lem}

\begin{proof}
Immediate from the definition of $\iota_a$ and $Y(\mathfrak{N})$.
\end{proof}

\section{Derived Hecke algebra of multiplicative group over number field}\label{section:hecke}
In this section, let us define the derived Hecke algebra and its action on the cohomology of torus $Y(\mathfrak{N})$. To do so, we follow the explicit description by Venkatesh \cite[\textsection 2]{venkatesh2016derived}.

Let $R$ be a ring.
For a prime $v$ of $F$, let $F_v$ be the $v$-adic completion of $F$, $O_v$ the valuation ring on $F_v$, 
$N(v)$ the absolute norm of $v$, and $\kappa_v$ the residue field of $v$. 

\subsection{Classical Hecke operator}
Let $v$ be a prime of $F$ coprime to $\mathfrak N$.
For $z\in F_v^\times/O_v^\times$, let $z:Y(\mathfrak{N})\rightarrow Y(\mathfrak{N})$ 
be the map induced by the multiplication $a\mapsto za$ with $z$. From this, we obtain the pullback map 
\begin{equation*}
z^*:H^\bullet(Y(\mathfrak{N}),R)\longrightarrow H^\bullet(Y(\mathfrak{N}),R).
\end{equation*}
Then the classical Hecke operator $h_{z,1}$ is defined by 
\[
h_{z,1}c:=z^*c\quad\text{for}\quad c\in H^\bullet(Y(\mathfrak N),R).
\]
Note that the map $F_v^\times/O_v^\times\to\mathrm{Aut}_RH^\bullet(Y(\mathfrak{N}),R)$ 
defined by $z\mapsto h_{z,1}$ is a group homomorphism.
From the equation $z\circ\iota_a=\iota_{za}$, we obtain the following formula:
\begin{equation}\label{eq:hecke}
\iota_a^*(h_{z,1}c)=\iota_{za}^*(c).    
\end{equation}

\subsection{Derived Hecke operator}
Let $v$ be a prime ideal of $F$ coprime to $\mathfrak{N}$.
We have the following canonical projection:
\begin{equation*}
\pi_{v,\mathfrak{N}}:Y(v\mathfrak{N})\longrightarrow Y(\mathfrak{N}),
\end{equation*}
which is a principal $\kappa_v^\times$-bundle, where the action of $k\in\kappa_v^\times$ on 
$(x,a)\in Y(v\mathfrak{N})$ is given by $k\cdot(x,a)=(x,ak^{-1})$. 
For a discrete group $G$, let $\mathcal{B}G$ be a classifying space of $G$.
By its universal property, $\pi_{v,\mathfrak{N}}$ corresponds to a map 
\begin{align}\label{eq:3.1}
\phi_{v,\mathfrak{N}}:Y(\mathfrak{N})\longrightarrow\mathcal{B}\kappa_v^\times,
\end{align}
which is determined uniquely up to homotopy. Assume further that $N(v)$ of $v$ is invertible in $R$.
Then the inflation map
\begin{equation*}
H^\bullet(\kappa_v^\times,R)\longrightarrow H^\bullet(O_v^\times,R)
\end{equation*}
is clearly an isomorphism. Thus, from now on, for $\alpha\in H^\bullet(O_v^\times,R)$, we denote by $\alpha\in H^\bullet(\kappa_v^\times,R)$ the image of $\alpha\in H^\bullet(O_v^\times,R)$ under the inverse of the inflation map by abusing the notation.
Let us denote by $\langle\cdot\rangle$ the composition of the following maps:
\begin{equation}\label{pullback:homotopy}
\begin{tikzcd}
H^\bullet(O_v^\times,R) \arrow{r}{\cong} &  H^\bullet(\kappa_v^\times,R) \arrow{r}{\cong} &  H^\bullet(\mathcal{B}\kappa_v^\times,R)\arrow{r}{\phi_{v,\mathfrak{N}}^*} & H^\bullet(Y(\mathfrak{N}),R).
\end{tikzcd}
\end{equation}
Note that $\langle\cdot\rangle:H^\bullet(O_v^\times,R)\to H^\bullet(Y(\mathfrak{N}),R)$ is a map of graded $R$-algebras since $H^\bullet(-,R)$ is a functor into the category of graded $R$-algebras.

\begin{defn}\label{derived:hecke:operator:definition}
For $\alpha\in H^\bullet(O_v^\times,R)$, 
let us define $h_{1,\alpha}\in\mathrm{End}_RH^\bullet(Y(\mathfrak{N}),R)$ by
\begin{equation*}
h_{1,\alpha}c:=\langle\alpha\rangle\cup c\quad\text{for}\quad c\in H^\bullet(Y(\mathfrak{N}),R),
\end{equation*} 
where $\cup$ is the cup product on the graded $R$-algebra $H^\bullet(Y(\mathfrak{N}),R)$. 
$h_{1,\alpha}$ is called {\it derived Hecke operator}.
\end{defn}
Note that the map $H^\bullet(O_v^\times,R)\to\mathrm{End}_RH^\bullet(Y(\mathfrak{N}),R)$ 
defined by $\alpha\mapsto h_{1,\alpha}$ is a map of graded $R$-algebras.

\begin{remark}\label{level:dividing:trivial}
We can also define the derived Hecke action for prime ideals $v$ of $F$ dividing $\mathfrak{N}$, but in the case of torus, such action is trivial for positive degree.
\end{remark}

Let $i_{v,\mathfrak{N}}:E(\mathfrak{N})\to\kappa_v^\times$ be the reduction map.
We use the same notation for the induced injection
$E(\mathfrak{N})/E(v\mathfrak{N})\hookrightarrow\kappa_v^\times$.
Let $\psi_{v,\mathfrak{N}}$ be the classifying map of the 
canonical principal $E(\mathfrak{N})/E(v\mathfrak{N})$-bundle
\begin{equation*}
E(v\mathfrak{N})\backslash F_{\mathbb{R},+}^\times/U
\longrightarrow E(\mathfrak{N})\backslash F_{\mathbb{R},+}^\times/U.
\end{equation*}

\begin{prop}\label{computation:hecke:operator}  
Let $\alpha\in H^\bullet(O_v^\times,R)$. Then,
\begin{equation*}
\iota_a^*(h_{1,\alpha}c)=i_{v,\mathfrak{N}}^*(\alpha)\cup\iota_{a}^*(c)
\quad \text{for}\quad c\in H^\bullet(Y(\mathfrak{N}),R).
\end{equation*}
\end{prop}

\begin{proof}
We give a brief proof; see \cite[Proposition 3.4]{KimKwon2024} for details.
The classifying maps give the following homotopy-commutative diagram:
\begin{equation}\label{cohomology:commutative:diagram}
\begin{tikzcd}
E(\mathfrak{N})\backslash F_{\mathbb{R},+}^\times/U
\arrow{r}{=}\arrow{d}{\simeq}
&E(\mathfrak{N})\backslash F_{\mathbb{R},+}^\times/U
\arrow{r}{\iota_a}\arrow{d}{\psi_{v,\mathfrak{N}}}
&Y(\mathfrak{N})\arrow{d}{\phi_{v,\mathfrak{N}}}\\
\mathcal{B}E(\mathfrak{N})\arrow{r}
&\mathcal{B}\frac{E(\mathfrak{N})}{E(v\mathfrak{N})}\arrow{r}{\mathcal{B}i_{v,\mathfrak{N}}}
&\mathcal{B}\kappa_v^\times.
\end{tikzcd}
\end{equation}
The bottom left arrow is induced by the canonical projection.
Each square follows from the map of bundles.
By Proposition~\ref{property:torus}, we identify the cohomology of
$E(\mathfrak{N})\backslash F_{\mathbb{R},+}^\times/U$ with $H^\bullet(E(\mathfrak{N}),R)$.
Applying $H^\bullet(-,R)$ to the diagram~\eqref{cohomology:commutative:diagram}, we obtain that
$\iota_a^*(\langle\alpha\rangle)=i_{v,\mathfrak{N}}^*(\alpha)$ under the identification.
Then, we are done since $\iota_a^*$ is a map of cohomology rings.
\end{proof}

By Proposition~\ref{computation:hecke:operator} and equation~\eqref{eq:hecke}, 
we obtain that
\begin{equation}\label{eq:comm}
h_{1,\alpha}h_{z,1}=h_{z,1}h_{1,\alpha}
\end{equation} 
for any $z\in F\times_v/O_v^\times$ and $\alpha\in H^\bullet(O_w^\times,R)$
for primes $v$ and $w$.

For $m\geq 0$ and $\underline{\alpha}=(\alpha_1,\ldots,\alpha_m)$ 
with $\alpha_j\in H^\bullet(O_{v_j}^\times,R)$, put
\[
h_{1,\underline{\alpha}}:=h_{1,\alpha_1}\cdots h_{1,\alpha_m},\quad
I_{\mathfrak{N}}(\underline{\alpha}):=
i_{v_1,\mathfrak{N}}^*(\alpha_1)\cup\cdots\cup i_{v_m,\mathfrak{N}}^*(\alpha_m).
\]
When $m=0$, we define $h_{1,\underline{\alpha}}$ and $I_{\mathfrak{N}}(\underline{\alpha})$ 
to be the identity operator and the identity element, respectively.

\begin{prop}\label{computation:hecke:operator:product}
Let $\underline{\alpha}$ be as above.
Then,
\begin{equation*}
\iota_a^*(h_{1,\underline{\alpha}}c)=I_{\mathfrak{N}}(\underline{\alpha})\cup\iota_{a}^*(c)
\quad \text{for} \quad c\in H^\bullet(Y(\mathfrak{N}),R).
\end{equation*}
\end{prop}

\begin{proof}
The case $m=0$ is immediate. For $m\geq 1$, use Proposition~\ref{computation:hecke:operator} repeatedly.
\end{proof}

\subsection{Derived Hecke algebra}
Let $v$ be a prime ideal of $F$ whose absolute norm is invertible in $R$.
Let us recall that the abstract derived Hecke algebra at $v$ with coefficients in $R$ is given by
\begin{equation*}
\mathcal{H}_{v,R}^\bullet=\mathrm{Ext}^\bullet_{C_{\rm c}(F_v^\times,R)}(R[F_v^\times/K(\mathfrak{N})_v],R[F_v^\times/K(\mathfrak{N})_v]).
\end{equation*}
Note that definition of $\mathcal{H}_{v,R}^0$ for general primes $v$ is in~\eqref{abstract:classcial:Hecke:algebra:def}.
There is an explicit description on the action of abstract derived Hecke algebras $\mathcal{H}^\bullet_{v,R}$ on the cohomology ring $H^\bullet(Y(\mathfrak{N}),R)$. 
We can easily check that $\mathcal{H}_{v,R}^j$ acts trivially on $H^\bullet(Y(\mathfrak{N}),R)$ if $j>0$ and $v\mid \mathfrak{N}$, as mentioned in Remark \ref{level:dividing:trivial}. Therefore, let us assume that $v$ is coprime to $\mathfrak{N}$.
By Venkatesh \cite[\textsection 2.4, (25)]{venkatesh2016derived}, we have the following isomorphism of $R$-modules:
\begin{equation}\label{explicit:description}
R[F_v^\times/O_v^\times]\otimes_R H^\bullet(O_v^\times,R)\cong \mathcal{H}^\bullet_{v,R}.
\end{equation}
Let us define an action of $\mathcal{H}^\bullet_{v,R}$ on $c\in H^\bullet(Y(\mathfrak{N}),R)$ by 
$(z,\alpha)\cdot c:= h_{1,\alpha}h_{z,1}c$, where $(z,\alpha)\in\mathcal{H}^\bullet_{v,R}$ is the image of $z\otimes\alpha\in R[F_v^\times/O_v^\times]\otimes_R H^\bullet(O_v^\times,R)$ under the isomorphism (\ref{explicit:description}).
Let us write the representation of the given action as follows: 
\begin{equation*}
\mathcal{H}^\bullet_{v,R}\to\mathrm{End}_R H^\bullet(Y(\mathfrak{N}),R),\ (z,\alpha)\mapsto h_{1,\alpha}h_{z,1},
\end{equation*}
which is clearly a map of graded $R$-algebras.
Let us denote by $\mathbb{T}_{v,R}^\bullet$ the image of $\mathcal{H}_{v,R}^\bullet$ in $\mathrm{End}_{R}H^\bullet(Y(\mathfrak{N}),R)$, which is a graded $R$-algebra.
Note that any pair of operators in $\mathbb{T}_{v,R}^\bullet$ for any distincts pair of primes $v$ graded commute in $\mathrm{End}_{R}H^\bullet(Y(\mathfrak{N}),R)$ by Venkatesh \cite[\textsection 2.10, Remark]{venkatesh2016derived}. 
From this, we may define the algebra generated by $\mathbb{T}_{v,R}^\bullet$ for primes $v$ of $F$. 
Let $T_0$ be the set of primes $v$ of $F$ such that $N(v)$ is invertible in $R$ and $v$ is coprime to $\mathfrak N$.

\begin{defn}[Derived Hecke algebra, {Venkatesh \cite[\textsection 2]{venkatesh2016derived}}]\label{definition:derived:hecke:algebra}
Let us denote by $\mathbb{T}_R^\bullet$ the graded $R$-subalgebra of $\mathrm{End}_{R}H^\bullet(Y(\mathfrak{N}),R)$
generated by $\mathbb{T}_{v,R}^\bullet$ for all $v\in T_0$
and $\mathbb T_{v,R}^0$ for all primes $v$ of $F$ coprime to $\mathfrak N$.
We call $\mathbb{T}_R^\bullet$ {\it the derived Hecke algebra} of $Y(\mathfrak{N})$ with coefficients in $R$.
\end{defn} 

\begin{remark}
In the torus case, adding classical Hecke operators at primes dividing $\mathfrak N$
does not enlarge $\mathbb T_R^0$.
Indeed, translation by an idele induces an action on cohomology depending only on
its narrow ray class, and every narrow ray class is represented by a prime coprime to
$\mathfrak N$ by the Chebotarev density theorem.
\end{remark}

For $a\in\mathbb{A}_F^\times$, define $1_a$ to be the characteristic function on $Y(\mathfrak{N})$ 
supported on $\mathrm{Im}(\iota_a)$. By Lemma \ref{representative:choice}, 
$1_a$ depends only on the ray class of modulus $\mathfrak{N}$ represented by $a$.
We may choose a complete representative set 
\begin{equation*}
C(\mathfrak{N})\subset\mathbb{A}_F^{(\infty),\times}
\end{equation*}
of $\mathrm{Cl}_F^+(\mathfrak{N})$ such that $1\in C(\mathfrak{N})$, 
where $1$ is the identity element of $\mathbb{A}_F^{(\infty),\times}$. 
By abusing the notation, for $a \in C(\mathfrak N)$ we let $1_a$ stand for $1_{(1,a)}$.
Then, $1_a$'s generate $H^0(Y(\mathfrak{N}),R)$ for $a\in C(\mathfrak{N})$.
Under the isomorphism $H^0(E(\mathfrak{N}),R)\cong R$, $\iota_b^*(1_a)=\delta_{ab}$ for $a,b\in C(\mathfrak{N})$.

Let $a$ be an element of $C(\mathfrak{N})$ and $z$ an element of 
$\mathbb A_F^\times/\widehat{\mathcal O}_F^\times$.
Let $z(a)$ be the element of $C(\mathfrak{N})$ such that 
$z(a)=za$ in $\mathrm{Cl}_F^+(\mathfrak{N})$. 
The operator $a\mapsto z(a)$ on $C(\mathfrak{N})$ defines the following permutation representation:
\begin{equation*}
\mathrm{Cl}_F^+(\mathfrak{N})\rightarrow\mathrm{Perm}\big(C(\mathfrak{N})\big),
\quad z\mapsto \big(a\mapsto z(a)\big).
\end{equation*}
Let us denote
\begin{equation*}
h_{z,1}:=\prod_{v\nmid\mathfrak N}h_{z_v,1}\in\mathbb{T}^0_R,
\end{equation*}
which is well defined as $\mathbb T_R^0$ is a commutative $R$-algebra
and $h_{z_v,1}$ is the identity operator for all but finitely many $v$.
Based on the fact that $1_a$ is a characteristic function, we obtain that
\begin{equation}\label{hecke:action:representation}
h_{z,1}1_a=1_{z^{-1}(a)}.
\end{equation}

Let us compute the $R$-module structure of $\mathbb{T}^0_R$ 
and the $\mathbb{T}^0_R$-module structure of $H^0(Y(\mathfrak{N}),R)$ and $\mathbb{T}^\bullet_R$:

\begin{prop}\label{Hecke:module:structure}
The following are true:
\begin{itemize}
\item[(1)] $H^0(Y(\mathfrak N),R)=\mathbb T_R^01_1$.
\item[(2)] $\mathbb{T}^0_R=\sum_{z\in \mathbb A_F^\times/\widehat{\mathcal O}_F^\times}Rh_{z,1}$.
\item[(3)] 
$
\mathbb{T}_R^\bullet
=\sum_{m\geq 0}\sum_{\underline{\alpha}}
\mathbb{T}_R^0 h_{1,\underline{\alpha}},
$
where $\underline{\alpha}$ runs over the $m$-copy of
\[
\bigsqcup_{v\in T_0}H^\bullet(O_v^\times,R).
\]
\end{itemize}
\end{prop}

\begin{proof}
The claim (1) follows from equation~\eqref{hecke:action:representation} for $a=1$
and the fact that $H^0(Y(\mathfrak N),R)$ is generated by $1_a$'s.
(2) follows from Lemma~\ref{representative:choice} and equation~\eqref{eq:hecke}.
For (3), $\mathbb{T}_R^\bullet$ is spanned over $R$ by finite products of operators $h_{1,\alpha}h_{z,1}$. 
Then equation~\eqref{eq:comm} allows us to move all degree-zero operators to the left. 
Their product belongs to $\mathbb{T}_R^0$, which proves (3).
\end{proof}

\begin{prop}\label{hecke:action:faithful}
The following map is injective:
\[
\mathbb T^\bullet_R\longrightarrow\operatorname{Hom}_R\left(H^0(Y(\mathfrak N),R),H^\bullet(Y(\mathfrak N),R)\right),
\quad
h\longmapsto h\vert_{H^0(Y(\mathfrak N),R)}.
\]
\end{prop}

\begin{proof}
Suppose that the restiction of $h\in\mathbb T_R^\bullet$ is zero.
Choose $a\in C(\mathfrak N)$ and $c\in H^\bullet(Y(\mathfrak N),R)$.
For $h=h_{z,1}h_{1,\underline\alpha}$, 
Proposition~\ref{computation:hecke:operator:product}
and equations~\eqref{eq:hecke},~\eqref{eq:comm},~\eqref{hecke:action:representation} give the equations
\[
\iota_a^*(hc)=I_{\mathfrak N}(\underline\alpha)\cup\iota_{z(a)}^*(c),
\quad
\iota_a^*(h1_b)=\delta_{z(a),b}I_{\mathfrak N}(\underline\alpha).
\]
From this, we obtain
\begin{equation}\label{component:formula}
\iota_a^*(hc)
=\sum_{b\in C(\mathfrak N)}\iota_a^*(h1_b)\cup\iota_b^*(c).
\end{equation}
By Proposition~\ref{Hecke:module:structure},
every $h\in\mathbb T_R^\bullet$ is an $R$-linear combination of $h_{z,1}h_{1,\underline\alpha}$.
Therefore, the formula~\eqref{component:formula} extends linearly to every $h\in\mathbb T_R^\bullet$.
From our assumption and Proposition~\ref{torus:copy}, we deduce that $hc=0$. 
As $h\in\mathbb T_R^\bullet$, we obtain that $h=0$.
\end{proof}

\begin{prop}\label{hecke:zero:free}
The map
\[
\mathbb T_R^0\to H^0(Y(\mathfrak N),R),\quad h\mapsto h1_1
\] 
is an isomorphism of $\mathbb T_R^0$-modules.
\end{prop}

\begin{proof}
The map of the statement is surjective 
by Proposition~\ref{Hecke:module:structure} (1). Choose $h\in\mathbb T_R^0$ and assume that $h1_1=0$.
By Proposition~\ref{Hecke:module:structure} and the commutativity of $\mathbb T^0_R$, 
we have $hc=0$ for any $c\in H^0(Y(\mathfrak N),R)$.
By Proposition~\ref{hecke:action:faithful}, we get $h=0$. 
Hence the map $h\mapsto h1_1$ is injective. So we are done.
\end{proof}

From these results, we obtain the property of $\Psi_R^\bullet$:

\begin{prop}\label{injectivity:psi}
The map $\Psi_R^\bullet$ defined in \eqref{eq:psi} is injective.
\end{prop}

\begin{proof}
By Proposition~\ref{hecke:zero:free}, the domain of $\Psi_R^\bullet$ can be identified with $\mathbb T_R^\bullet$.
Under this identification, the map $\Psi_R^\bullet$ can be identified with the map $h\mapsto h1_1$.
Suppose that $h1_1=0$. By using Proposition~\ref{Hecke:module:structure} and~\eqref{eq:comm}, 
we obtain that $hc=0$ for each $c\in H^0(Y(\mathfrak N),R)$. 
By Proposition~\ref{hecke:action:faithful}, we are done.
\end{proof}

\section{Non-vanishing mod {\it p} of derived Hecke algebras}\label{section:nonvanishing}
Let $p$ be a rational prime. We take $R$ to be $\mathbb{F}_p$, the finite field with $p$ elements.
Then, $T_0$ becomes the set of primes $v$ of $F$ coprime to $p\mathfrak{N}$.
Note that $H^j(O_v^\times,\mathbb{F}_p)\cong H^j(\kappa_v^\times,\mathbb{F}_p)$ if $v\in T_0$. 
Define $T_1\subset T_0$ to be the set of primes $v$ in $T_0$ such that $p\mid N(v)-1$.
We can observe from the Chebotarev density theorem that $T_1$ has a positive proportion in the set of primes of $F$. 
For $j\geq 0$, we put
\begin{equation*}
\mathbb{T}^j:=\mathbb{T}_{\mathbb{F}_p}^j\subset\mathrm{End}_{\mathbb{F}_p}H^\bullet(Y(\mathfrak{N}),\mathbb{F}_p).
\end{equation*}

Recall the map $\Psi^\bullet$ is the map $\Psi^\bullet_R$ for $R=\mathbb F_p$.
Now we prove that $\Psi^\bullet$ is a $\mathbb{T}^0$-linear isomorphism.
Define $d_p$ to be the $\mathbb{F}_p$-dimension of the graded subalgebra of 
$H^\bullet(E(\mathfrak{N}),\mathbb{F}_p)$ generated by the images 
$i_{v,\mathfrak{N}}^*H^\bullet(\kappa_v^\times,\mathbb{F}_p)$ for all $v\in T_0$.
Note that $H^{\bullet}(\kappa_v^\times,\mathbb{F}_p)=H^0(\kappa_v^\times,\mathbb{F}_p)$ if $v\in T_0\setminus T_1$.
Let $h_F^+(\mathfrak{N})$ be the cardinality of $\mathrm{Cl}_F^+(\mathfrak{N})$.
We next compute the dimension of the image of $\Psi^\bullet$:

\begin{lem}\label{dimension:prop}
The image of $\Psi^\bullet$ has $\mathbb{F}_p$-dimension $h_F^+(\mathfrak{N})\cdot d_p$.
\end{lem}

\begin{proof} 
By Proposition~\ref{Hecke:module:structure} (1), 
the image $\operatorname{Im}(\Psi^\bullet)$ coincides with $\mathbb{T}^\bullet 1_1$.
By Proposition \ref{Hecke:module:structure} and equation~\eqref{eq:comm},
\[
\mathbb{T}^\bullet 1_1
=\sum_{m\geq 0}\sum_{\underline\alpha}\mathbb{T}^0 h_{1,\underline\alpha}1_1
=\sum_{a\in C(\mathfrak{N})}\sum_{m}\sum_{\underline\alpha}\mathbb{F}_p h_{1,\underline\alpha}1_a,
\]
where $\underline\alpha$ runs over the $m$-copy of
$\bigsqcup_{v\in T_1}H^\bullet(\kappa_v^\times,\mathbb{F}_p)$.
The summands are linearly disjoint for $a\in C(\mathfrak N)$,
as $h_{1,\underline\alpha}1_a$'s are cohomology classes supported on the distinct connected components.
Therefore, we obtain from the above equation that
\begin{equation*}
\dim_{\mathbb{F}_p}\operatorname{Im}(\Psi^\bullet)
=h_F^+(\mathfrak{N})\cdot\dim_{\mathbb{F}_p}\sum_{m\geq 0}
\sum_{\underline\alpha}\mathbb{F}_p h_{1,\underline\alpha}1_1.
\end{equation*}
By Proposition~\ref{computation:hecke:operator:product}, we have
\begin{equation*}
\iota_b^*(h_{1,\underline{\alpha}}1_a)=\delta_{a,b}I_{\mathfrak{N}}(\underline{\alpha}).
\end{equation*}
Thus under the injection $\iota_1^*$ (Proposition~\ref{torus:copy}), the space on the right maps to
\begin{equation*}
\sum_{m\geq0}\sum_{\underline{\alpha}}\mathbb{F}_p I_{\mathfrak{N}}(\underline{\alpha})
\subset H^\bullet(E(\mathfrak{N}),\mathbb{F}_p)
\end{equation*}
by Proposition~\ref{computation:hecke:operator:product}. Its dimension is $d_p$ by definition. So we are done.
\end{proof}

Let $r_p$ denote the $\mathbb F_p$-dimension of $\operatorname{Hom}(\mathcal{O}_F^\times, \mathbb F_p)$. Note that $r_p=r$ unless $F$ contains a primitive $p$-th root of unity, in which case $r_p=r+1$.
Denote by $\delta_p$ the $\mathbb F_p$-dimension of $\mathrm{Hom}(\mathcal{O}_F^\times/E(\mathfrak{N}),\mathbb{F}_p)$.
Define $t_p\geq 0$ to be the dimension of the following subspace
\begin{equation*}
\sum_{v\in T_1}i_{v,\mathfrak{N}}^*\mathrm{Hom}(\kappa_v^\times,\mathbb{F}_p)\subset\mathrm{Hom}(E(\mathfrak{N}),\mathbb{F}_p).
\end{equation*}

\begin{lem}\label{dimension:tp:dp}
 $2^{t_p}\leq d_p\leq 2^r$.
\end{lem}

\begin{proof}
Since $E(\mathfrak{N})$ is free abelian of rank $r$, 
its cohomology ring is isomorphic to the exterior algebra on $H^1(E(\mathfrak{N}),\mathbb{F}_p)$ 
and has dimension $2^r$.
The subspace defining $t_p$ generates an exterior subalgebra of dimension $2^{t_p}$, 
contained in the graded subalgebra defining $d_p$. So we are done.
\end{proof}

\begin{lem}\label{dimension:lemma}
There is a finite subset $S\subset T_1$ such that the space
\[
\sum_{v\in S}i_{v,\mathfrak N}^*\mathrm{Hom}(\kappa_v^\times,\mathbb{F}_p)
\]
is equal to the space defining $t_p$ and has dimension $r_p-\delta_p$.
\end{lem}

\begin{proof}
Denote by $i_v\colon \mathcal{O}_F^\times \to \kappa_v^\times$ the reduction map.
Then $i_{v,\mathfrak N}$ is the restriction of $i_{v}$.
Thus, for a finite subset $T$ of $T_1$, the map
\begin{equation*}
\prod_{v\in T}\mathrm{Hom}(\kappa_v^\times,\mathbb{F}_p)
\longrightarrow\mathrm{Hom}(E(\mathfrak{N}),\mathbb{F}_p),
\quad
(\phi_v)_{v\in T}\longmapsto
\sum_{v\in T}i_{v,\mathfrak N}^*(\phi_v)
\end{equation*}
factors through $\mathrm{Hom}(\mathcal{O}_F^\times,\mathbb{F}_p)$ 
via the map $(\phi_v)_{v\in T}\mapsto\sum_{v\in T}i_v^*(\phi_v)$, 
followed by restriction to $E(\mathfrak N)$.
By \cite[Lemma 4.6]{KimKwon2024}, which is an isomorphism theorem 
induced by Grunwald--Wang theorem, there is a finite subset
$S\subset T_1$ such that the map
$(\phi_v)_{v\in S}\mapsto\sum_{v\in S}i_v^*(\phi_v)$
is an isomorphism onto
$\mathrm{Hom}(\mathcal{O}_F^\times,\mathbb{F}_p)$.
Using the exact sequence
\begin{equation*}
\begin{tikzcd}
0 \arrow{r} & \mathrm{Hom}(\mathcal{O}_F^\times/E(\mathfrak{N}),\mathbb{F}_p) \arrow{r} &  \mathrm{Hom}(\mathcal{O}_F^\times,\mathbb{F}_p) \arrow{r} & \mathrm{Hom}(E(\mathfrak{N}),\mathbb{F}_p),
\end{tikzcd}
\end{equation*}
the assertion of the lemma follows.
\end{proof}

\begin{remark}
Since $\mathcal{O}_F^\times/E(\mathfrak{N})$ is a finite abelian group, $\delta_p>0$ if and only if $p$ divides $\left |\mathcal{O}_F^\times/E(\mathfrak{N})\right |$. In particular, when $r>0$ and $p$ does not divide $\left | \mathcal{O}_F^\times/E(\mathfrak{N}) \right|$, we have $r_p-\delta_p >0$. 
\end{remark}

Now we can prove the following theorems:
\begin{thm}\label{spectral:degeneracy}
Let $\lambda$ be a Hecke eigensystem on $H^0(Y(\mathfrak{N}),\overline{\mathbb{F}}_p)$.
For every integer $1\leq j\leq r_p-\delta_p$, the restriction 
\[
\mathbb T^j\otimes_{\mathbb T^0} H^0(Y(\mathfrak{N}),\overline{\mathbb{F}}_p)[\lambda]
\longrightarrow
H^j(Y(\mathfrak{N}),\overline{\mathbb{F}}_p)[\lambda],\quad h\otimes c\longmapsto hc
\]
of $\Psi^j$ is well-defined and non-vanishing, where $H^i[\lambda]$'s are the $\lambda$-eigenspaces.
\end{thm}

\begin{proof}
We consider the derived Hecke action with coefficients in $\overline{\mathbb{F}}_p$.
The result will be equal to the base change of the derived Hecke action with $\mathbb F_p$-coefficients, as explained in \cite[\textsection 2.12]{venkatesh2016derived}.
Also note that the natural map induces an isomorphism $H^j(Y(\mathfrak{N}),\overline{\mathbb{F}}_p)\simeq H^j(Y(\mathfrak{N}),\mathbb{F}_p)\otimes_{\mathbb{F}_p}\overline{\mathbb{F}}_p$, because $\mathbb F_p \to \overline{\mathbb{F}}_p$ is flat.

Let $\lambda:\mathbb{T}^0\rightarrow\overline{\mathbb{F}}_p$ be a Hecke eigensystem on 
$H^0(Y(\mathfrak{N}),\overline{\mathbb{F}}_p)$ and $c\in H^0(Y(\mathfrak{N}),\overline{\mathbb{F}}_p)$ 
an eigenvector for the given Hecke eigensystem. 
Note that $c$ can be written as a linear combination of $1_a$'s over $\overline{\mathbb{F}}_p$. 
Let $1\leq j\leq r_p-\delta_p$.
By Lemma~\ref{dimension:lemma}, we can choose $v_1,\ldots,v_j\in T_1$ and 
$\alpha_i\in H^1(\kappa_{v_i}^\times,\mathbb{F}_p)$ such that
$i_{v_1,\mathfrak{N}}^*(\alpha_1),\ldots,i_{v_j,\mathfrak{N}}^*(\alpha_j)$
are linearly independent in $H^1(E(\mathfrak{N}),\mathbb{F}_p)$.
Since $E(\mathfrak{N})$ is free, its cohomology ring is isomorphic to the exterior algebra on 
$H^1(E(\mathfrak{N}),\mathbb{F}_p)$. Thus, for $\underline{\alpha}=(\alpha_1,\ldots,\alpha_j)$,
$I_{\mathfrak{N}}(\underline{\alpha})\neq 0.$
Choose $a\in C(\mathfrak{N})$ with $\iota_a^*(c)\neq 0$.
By Proposition~\ref{computation:hecke:operator:product},
\begin{equation*}
\iota_a^*(h_{1,\underline{\alpha}}c)
=I_{\mathfrak{N}}(\underline{\alpha})\cup\iota_a^*(c)\neq 0.
\end{equation*}
Hence $h_{1,\underline{\alpha}}c$ is a non-zero class in $H^j(Y(\mathfrak{N}),\overline{\mathbb{F}}_p)$.
Then for any $h_0\in\mathbb{T}^0$, Proposition~\ref{Hecke:module:structure} and equation~\eqref{eq:comm} give
$h_0h_{1,\underline{\alpha}}c
=h_{1,\underline{\alpha}}h_0c
=\lambda(h_0)h_{1,\underline{\alpha}}c$.
So we are done.
\end{proof}

\begin{thm}\label{modp:nonvanishing}
The dimension $t_p$ is equal to $r_p-\delta_p$. 
In particular, the map $\Psi^\bullet$ is an isomorphism if and only if $\delta_p=r_p-r$.
\end{thm}

\begin{proof}
The first assertion of the theorem follows from Lemma~\ref{dimension:lemma}. 

Suppose first that $\delta_p=r_p-r$. Then $t_p=r$, so Lemma~\ref{dimension:tp:dp} gives $d_p=2^r$.
By Lemma~\ref{dimension:prop}, $\Psi^\bullet$ is surjective, hence an isomorphism by Proposition~\ref{injectivity:psi}.

Conversely, suppose that $\delta_p\neq r_p-r$.
Since $t_p\leq r$, the first assertion implies $t_p<r$.
The dimension formula for $\Psi^1$ was proved in \cite[Lemma 4.3]{KimKwon2024} 
by an argument analogous to the proof of Lemma~\ref{dimension:prop}.
It gives
\begin{equation*}
\dim_{\mathbb{F}_p}\operatorname{Im}(\Psi^1)
=h_F^+(\mathfrak{N})\cdot t_p
<h_F^+(\mathfrak{N})\cdot r
=\dim_{\mathbb{F}_p}H^1(Y(\mathfrak{N}),\mathbb{F}_p).
\end{equation*}
Therefore $\Psi^1$ is not surjective.
Since $\Psi^\bullet$ preserves degrees, it cannot be an isomorphism.
\end{proof}

\begin{remark}\label{remark:dimension}
For the condition $\delta_p=r_p-r$ in Theorem\,\ref{modp:nonvanishing} to hold true, 
we have either $\delta_p=r_p-r=1$ or $\delta_p=r_p-r=0$. 
The latter happens exactly when $p$ does not divide $\left|\mathcal{O}_F^\times / E(\mathfrak N)\right|$. 
The former happens exactly when $F$ contains a primitive $p$-th root of unity and $\delta_p=1$. 
\end{remark}

Since localization is exact, Theorem~\ref{modp:nonvanishing} implies:
\begin{cor}\label{cor:4.10}
Let $\mathfrak{m}$ be a maximal ideal of $\mathbb{T}^0$. If $p$ does not divide $|\mathcal{O}_F^\times/E(\mathfrak{N})|$, we have the following isomorphism:
\begin{equation*}
\mathbb{T}^\bullet_\mathfrak{m}\otimes_{\mathbb{T}^0_\mathfrak{m}}H^0(Y(\mathfrak{N}),\mathbb{F}_p)_\mathfrak{m}
\cong H^\bullet(Y(\mathfrak{N}),\mathbb{F}_p)_\mathfrak{m}.
\end{equation*}
\end{cor}
After extending scalars to $\overline{\mathbb F}_p$, each cohomology group is a finite-dimensional
module over the finite-dimensional degree-zero Hecke algebra.
Its localization at the maximal ideal corresponding to $\lambda$ is nonzero
if and only if its $\lambda$-eigenspace is nonzero.

We conclude this section by deducing our main results. Assume $p$ does not divide the order of $\mathcal{O}_F^\times/E(\mathfrak{N})$.
It implies 
\begin{equation}\label{dimension:equation:remark}
\delta_p=r_p-r=0
\end{equation}
due to Remark \ref{remark:dimension}.

\begin{proof}[Proof of Theorem \ref{thm:A}]
Theorem \ref{modp:nonvanishing}, combined with the equation \eqref{dimension:equation:remark}, implies Theorem~\ref{thm:A}.
\end{proof}

\begin{proof}[Proof of Theorem \ref{thm:B}]
Recall that we assumed that $F$ is neither $\mathbb Q$ nor an imaginary quadratic field. In particular, $r_p>0$.
Now Theorem~\ref{spectral:degeneracy} and Corollary~\ref{cor:4.10}, combined with the equation \eqref{dimension:equation:remark}, imply Theorem~\ref{thm:B}. 
\end{proof}

\begin{proof}[Proof of Theorem~\ref{thm:C}]
Theorem~\ref{thm:A}, combined with Proposition~\ref{hecke:zero:free}, implies Theorem~\ref{thm:C}.
\end{proof}

\bibliographystyle{plain}
\bibliography{ref}

\end{document}